\documentclass[11pt,a4paper,twoside,draft]{article}

\usepackage[T1]{fontenc}
\usepackage{graphicx}
\usepackage{amsmath,amsthm,amsfonts}

\newtheorem{thm}{Theorem}[section]
\newtheorem{prop}[thm]{Proposition}
\newtheorem{lem}[thm]{Lemma}

\newtheorem{deff}[thm]{Definition}
\newtheorem{rem}[thm]{Remark}

\newtheoremstyle{note}{3pt}{3pt}{\rm}{}
{\bf}{.}{.5em}{}
\theoremstyle{note}

\newcommand{\R}{\mathbb R}

\title{Simultaneously Non-convergent Frequencies of Words in Different Expansions}
\author{David F\"arm\\
\scriptsize{Centre for Mathematical Sciences}\\
\scriptsize{Lund University, Sweden}\\
\scriptsize{david@maths.lth.se}}

\begin{document}

\maketitle

\begin{abstract}
\noindent We consider expanding maps such that the unit interval can be represented as a full symbolic shift space with bounded distortion.  There are already theorems about the Hausdorff dimension for sets defined by the set of accumulation points for the frequencies of words in one symbolic space at a time. It is shown in this text that the dimension is preserved when sets defined using different maps are intersected. More precisely, it is proven that the dimension of any countable intersection of sets defined by their sets of accumulation for frequencies of words in different expansions, has dimension equal to the infimum of the dimensions of the sets that are intersected. As a consequence, the set of numbers for which the frequencies do not exist has full dimension even after countable intersections. We prove these results also for a dense set of $\beta$-shifts. 
\end{abstract}

\section{Introduction}\label{introduction}

\subsection{Expanding maps generating full shifts} \label{fullshift}

Let $f\colon [0,1)\mapsto [0,1)$ be such that $[0,1)$ can be split into a finite number $g_f$ of intervals $[a,b)$ such that $f|_{[a,b)}$ is monotone and onto for each of these intervals. We take an enumeration of the intervals and associate each interval to the corresponding number so that we can refer to an interval as $[n]$ where $n$ is the appropriate number. Assume that for each of the intervals $[a,b)$ it holds that  $|f(x)-f(y)|\geq |x-y|$ for all $x,y\in [a,b)$. Then we can define cylinders 
\[
C_{x_1 \dots x_n}:=\Big\{\, x\in [0,1):\bigcap_{i=1}^n f^{-(i-1)}(x)\in [x_i]  \, \Big\},
\]
where $C_{x_1 \dots x_n}$ is called a generation $n$ cylinder. We will consider $[0,1)$ as a generation $0$ cylinder. Assume that $\lim_{n\to \infty}|C_{x_1 \dots x_n}|= 0$ for all 
\[
(x_i)_{i=1}^\infty \in \Sigma_{g_f}:=\{0,1,\dots, g_f-1\}^{\mathbb N}.
\]
Then we have a unique correspondence between points $x\in [0,1]$ and sequences $(x_i)_{i=1}^\infty \in \Sigma_{g_f}$.

For a given integer $m>0$ we consider words $w= i_1,\dots,i_m$ of length $m$ in the alphabet $\{0,\dots,g_f-1\}$. We can enumerate these words as $\{w_j\}_{j=1}^{g_f^m}$. For any number $x\in[0,1]$ and $n>m$, let 
\[
\tau_{w_j}^f(x,n)=\#\big\{\, i\in\{1,\dots,n-m\}: x_i,\dots,x_{i+m-1}=w_j \, \big\}
\]
and 
\begin{align}
G_{\bar p}^{f,m}=\left\{\, x:\frac{\tau_{w_j}^f(x,n)}{n-m}\to p_{w_j}, \quad n\to\infty,\quad  j=1, \dots, g_f^m\, \right\}\label{defofg},
\end{align}
where $\bar p=(p_{w_1},\dots p_{w_{g_f^m}})$ such that $0\leq p_{w_i}\leq 1$ for all $i$ and $\sum_{i=1}^{g_f^m}p_{w_i}=1$. Here $\bar p$ can be interpreted as the frequencies with which the words of length $m$ occur. Note that for many $x\in [0,1]$, the limit in (\ref{defofg}) does not even exist. Consider the sets
\[
G_{\bar p}^{f,m}(n,\epsilon)=\left\{\, x\in [0,1): p_{w_j}-\epsilon< \frac{\tau_{w_j}^f(x,n)}{n-m}<p_{w_j}+\epsilon, \forall j\, \right\}.
\] 
Note that $G_{\bar p}^{f,m}(n,\epsilon)$ is the union of of all generation $n$ cylinders for which the frequencies of words of length $m$ in the finite sequence determining the cylinder is $\epsilon$ close to $\bar p$. If a point $x$ is at the left endpoint of one of these cylinders, then $\frac{\tau_{w}^f(x,n)}{n-m}\to 0$ for all words $w$ of length $m$ except $0^m$. Thus, if $p_w>\epsilon$ for some word $w\neq 0^m$, then $x$ cannot be in $G_{\bar p}^{f,m}(n+N,\epsilon)$
when $N$ gets too large. This leads to the following conclusion.

\begin{rem}\label{gopen}
Let $m\in \mathbb N$ and $\epsilon>0$. If $\bar p$ is such that for some word $w\neq 0^m$ of length $m$ we have $p_w>\epsilon$, then the set 
\[
\bigcap_{N=1}^\infty \bigcup_{n=N}^\infty  G_{\bar p}^{f,m}(n,\epsilon)
\]
is a $G_\delta$ set.
\end{rem}

Note also that
\[
G_{\bar p}^{f,m} \subset \bigcap_{N=1}^\infty \bigcup_{n=N}^\infty  G_{\bar p}^{f,m}(n,\epsilon)
\]
for all $\epsilon >0$.

\begin{deff}
A number $x\in [0,1]$ is $m$-normal to $f$ if for all words $w$ of length $m$ we have $\frac{\tau^f_w(x,n)}{n-m}\to |C_w|$ as $n \to \infty$, where $|C_w|$ denotes the length of the set $C_w$.
\end{deff}

It is easy to see that if 
\[
x\in \bigcap_{N=1}^\infty \bigcup _{n=N}^\infty G_{\bar p}^{f,m}(n,\epsilon),
\]
where $0<\epsilon<||C_{w_j}|-p_{w_j}|$ for some $j$, then the frequency of the word $w_j$ of length $m$ in the expansion of $x$ cannot be $|C_{w_j}|$. Thus, $x$ is not $m$-normal to $f$. In fact, the expressions for the frequencies in (\ref{defofg}) do not even have to converge. For $x$ to be in $\limsup_{n\to \infty} G_{\bar p}^{f,m}(n,\epsilon)$, it is sufficient that the vector $\bar p$ is a point of accumulation for these expressions.

Let $A^{f,m}(x)$ denote the set of points of accumulation in $[0,1]^{g_f^m}$ for the frequencies $\Big(\frac{\tau_{w_i}^f(x,n)}{n-m}\Big)_{i=1}^{g_f^m}$ as $n\to \infty$, where $(w_i)_{i=1}^{g_f^m}$ is an enumeration of the words of length $m$. We will need the following conditions on $f$.

\begin{enumerate}
\item[(i)]
Assume that we have bounded distortion, \mbox{i.e.} that there is a constant $K_f>0$ such that for any cylinder $C_{x_1 \dots x_n}$, including $[0,1)$, it holds that 
\[
\frac{|(f^n)'(y)|}{|(f^n)'(z)|}<K_f
\]
for all $y,z\in C_{x_1 \dots x_n}$, which implies that
\[
\frac{|C_{x_1\dots x_{n+1}}|}{|C_{x_1 \dots x_n}|}>\frac{1}{g_f K_f}
\]
for all sequences $(x_i)_{i=1}^\infty \in \Sigma_{g_f}$ and all $n \in \mathbb N$.

\item[(ii)]
Given $f$ and $m\in \mathbb N$ there is a vector $\bar p$ such that $\dim_H(G_{\bar p}^{f,m})=1$. Assume that $f$ is such that for each word $w$ of length $m$ there is a vector $\bar q$ such that $q_w\neq p_w$, for which  $\dim_H(G_{\bar q}^{f,m})$ is arbitrarily close to $1$.

\end{enumerate}

We will prove the following theorems.

\begin{thm}\label{nonnormalthm}
Let $(f_i)_{i=1}^\infty$ be a sequence of functions satisfying conditions $(i)$ and $(ii)$. Then the set of numbers that are not $m$-normal to any of these $f_i$ for any $m$ has Hausdorff dimension $1$. 
\end{thm}

\begin{thm}\label{extremelynonormal}
Let $(f_i)_{i=1}^\infty$ be a sequence of functions satisfying conditions $(i)$ and $(ii)$. Then the set of numbers for which the frequency $\frac{\tau_{w}^f(x,n)}{n-m}$ does not converge as $n\to \infty$ for any word $w$ of any length $m$, in the expansion to any of the functions $f_i$, has Hausdorff dimension 1. 
\end{thm}

\begin{thm}\label{intersectiontheorem}
Let $(f_i)_{i=1}^\infty$ be a sequence of functions satisfying condition $(i)$, let $(m_i)_{i=1}^\infty $ be a sequence of numbers in $\mathbb N$ and let $(\bar p_i)_{i=1}^\infty$ be a sequence such that $\bar p_i\in [0,1]^{g_{f_i}^{m_i}}$ for each $i$. Then
\[
\dim_H\left(\, \bigcap_{i=1}^\infty \{\, x : \bar p_i \in A^{f_i,m_i}(x)\, \}\, \right)=
\inf_i \Bigg\{ \, \dim_H \Big( \{\, x:\bar p_i \in A^{f_i,m_i}(x)\,  \} \Big) \, \Bigg\}.
\]
\end{thm}

\subsection{$\boldsymbol{\beta}$-shifts where the expansion of 1 terminates}

The following method to expand real numbers in non-integer bases was introduced by R\'enyi~\cite{Renyi} and Parry~\cite{Parry}. For more details and proofs of the statements below, see their articles. 

Let $[x]$ denote the integer part of the number $x$. Let $\beta \in (1,2)$. For any $x \in [0, 1]$ we associate the sequence $d(x,\beta) =  \{d_n (x, \beta)\}_{n=0}^\infty \in \{0, 1\}^\mathbb N$ defined by
    \[
      d_n (x, \beta) :=
      [\beta f_\beta^{n} (x)],
    \]
    where $f_\beta (x) = \beta x \mod 1$.
    The closure of the set
    \[
      \{\, d (x,\beta) : x \in [0,1)\,\}
    \]
    is denoted by $S_\beta$ and it is called the $\beta$-shift. It is invariant under the left-shift $\sigma \colon \{i_n\}_{n=0}^\infty \mapsto \{i_{n+1} \}_{n=1}^\infty$ and the map $d (\cdot, \beta) \colon x \mapsto d (x, \beta)$ satisfies $\sigma^n ( d (x, \beta) ) = d ( f_\beta^n (x), \beta)$. If we order $S_\beta$ with the lexicographical ordering then the map $d ( \cdot, \beta)$ is  one-to-one and monotone increasing. The subshift $S_\beta$ satisfies
    \begin{equation} \label{eq:Sbeta}
      S_\beta = \{\, \{j_k\} : \sigma^n \{j_k\} < d (1, \beta) \ \forall n \,\}.
    \end{equation}
  
     If $x \in [0,1]$ then
    \[
      x = \sum_{k=0}^\infty \frac{d_k (x, \beta) }{\beta^{k + 1}}.
    \]
    We let $\pi_\beta$ be the map $\pi_\beta \colon S_\beta \to [0,1)$ defined by
    \[
      \pi_\beta \colon \{i_k\}_{k=0}^\infty \quad \mapsto \quad \sum_{k=0}^\infty \frac{i_k}{\beta^{k + 1}}.
    \]
    Hence, $\pi_\beta ( d(x, \beta)) = x$ holds for any $x \in [0,1)$ and $\beta > 1$.

    A cylinder $s$ is a subset of $[0,1)$ such that 
    \[
      s  :=\pi_\beta( \{\, \{j_k\}_{k=0}^\infty : i_k = j_k,\ 0\leq k < n \,\})
    \]
    holds for some $n$ and some sequence $\{i_k\}_{k=0}^\infty$. We then say that $s$ is an $n$-cylinder or a cylinder of generation $n$ and write
    \[
      s = [i_0 \cdots i_{n-1}].
    \]

As in Section \ref{fullshift} we define     
\[
\tau_{w}^\beta(x,n)=\#\big\{\, i\in\{0,\dots,n-m-1\}: x_i,\dots,x_{i+m-1}=w \, \big\}
\]
for any word $w$ of length $m$ and

\begin{align}
G_{\bar p}^{\beta,m}=\left\{\, x:\frac{\tau_{w_j}^\beta(x,n)}{n-m}\to p_{w_j}, \quad n\to\infty, \quad j=1, \dots, 2^{m} \, \right\}
\end{align}
where $0\leq p_{w_j}\leq 1$ for all $j$. We also define
\[
G_{\bar p}^{\beta,m} (n,\epsilon) =\left\{\, x\in [0,1): p_w-\epsilon< \frac{\tau_{w_j}^f(x,n)}{n-m}<p_w+\epsilon \quad j=1, \dots, 2^m\, \right\}.
\] 
As in Section \ref{fullshift} we note the following.

\begin{rem}\label{betagopen}
Let $m\in \mathbb N$ and $\epsilon>0$. If $\bar p$ is such that for some word $w\neq 0^m$ of length $m$ we have $p_w>\epsilon$, then the set 
\[
\bigcap_{N=1}^\infty \bigcup_{n=N}^\infty  G_{\bar p}^{\beta,m}(n,\epsilon)
\]
is a $G_\delta$ set.
\end{rem}

Note that
\[
G_{\bar p}^{\beta,m} \subset \bigcap_{N=1}^\infty \bigcup_{n=N}^\infty  G_{\bar p}^{\beta,m} (n,\epsilon)
\]
for all $\epsilon >0$. Let $A^{\beta,m}(x)$ denote the set of points of accumulation in $[0,1] R^{2^m}$ for the frequencies $\Big(\frac{\tau_{w_i}^\beta(x,n)}{n-m}\Big)_{i=1}^{2^m}$ as $n\to \infty$, where $(w_i)_{i=1}^{2^m}$ is an enumeration of the words of length $m$.     
    
Consider $\beta$ such that the expansion of 1 terminates, \mbox{i.e.} such that we have $d(1, \beta)=j_0 \dots j_{k-1} 0^\infty$. The set of such $\beta$ is dense in $(1,2)$ and for such $\beta\in (1,2)$ we can use (\ref{eq:Sbeta}) to construct $S_\beta$ from the full shift $\Sigma_2=\{0,1\}^{\mathbb N}$ as follows. There are finitely many words $w$ of length $k$ such that $w<d(1, \beta)$. If we start with $\Sigma_2$ and remove all elements that contain any of these words, then by  (\ref{eq:Sbeta})  we get $S_\beta$. Thus $S_\beta$ is a subshift of finite type. For such shifts there is a finite constant $C_\beta>0$ such that
\[
\frac{|C_{x_0\dots x_{n}}|}{|C_{x_0 \dots x_{n-1}}|}>\frac{1}{\beta C_\beta}
\]
for all sequences $(x_i)_{i=0}^\infty \in S_{\beta}$ and all $n \in \mathbb N$. We can use this to prove the following theorems.

\begin{thm}\label{betaextremelynonormal}
Let $(\beta_i)_{i=1}^\infty$ be any sequence in $(1,2)$ such that the expansion of $1$ terminates for each $\beta_i$.  Then the set of numbers for which $\frac{\tau_{w}^\beta(x,n)}{n-m}$ does not converge as $n\to \infty$ for any word $w$ of any length $m$, in the expansion to any of the functions $f_{\beta_i}$, has Hausdorff dimension $1$. 
\end{thm}
    
\begin{thm}\label{betaintersectiontheorem}
Let $(\beta_i)_{i=1}^\infty$ be any sequence in $(1,2)$ such that the expansion of $1$ terminates for each $\beta_i$, let $(m_i)_{i=1}^\infty $ be a sequence of numbers in $\mathbb N$ and let $(\bar p_i)_{i=1}^\infty$ be a sequence such that $\bar p_i\in [0,1]^{2^{m_i}}$ for each $i$. Then
\[
\dim_H\left(\, \bigcap_{i=1}^\infty \{\, x : \bar p_i \in A^{\beta_i,m_i}(x)\, \}\, \right)=
\inf_i \Bigg\{ \, \dim_H \Big( \{\, x:\bar p_i \in A^{\beta_i,m_i}(x)\,  \} \Big) \, \Bigg\}.
\]
\end{thm}

\section {Falconer's classes}

In \cite{falconer}, Falconer defines classes of sets in $\R^n$ with the property that dimensions are preserved under countable intersections. The idea in the proofs of the main theorems of this text is to show that the sets involved are in the classes defined by Falconer. We present here a one-dimensional version of these classes. 

\begin{deff}\label{defofclass}
For $0<s\leq1$, let $\mathcal G^s$ be the class of $G_\delta$ sets $F\subset \R$ such that $\dim_H \left(\cap_{i=1}^\infty f_i(F)\right)\geq s$ for all sequences of similarity transformations $\{f_i\}_{i=1}^\infty$. 
\end{deff}

\begin{rem}\label{dimensionequal}
It follows immediately from the definition that for each choice of $s \in (0,1]$ and $t\in (0,s)$ we have
\[
\mathcal G^t\subset \mathcal G^s \textrm{ and } \ \mathcal G^s=\bigcap_{t\in (s,1]}\mathcal G^t.
\]
\end{rem}

\noindent As a tool in his proofs, Falconer uses outer measures $M_\infty^s$ defined by
\[
M_\infty^s(F)=\inf \Big\{\, \sum_{i=1}^\infty |I_i|^s:F\subset \cup_{i=1}^\infty I_i \, \Big\}
\]
where each $I_i$ is of the form $[2^km,2^k(m+1))$, $m\in \{0, \dots 2^m-1\}$, with which we call dyadic intervals. He proves that $\mathcal G^s$ can be characterised in several ways. We present here the characterisation we will use.

\begin{thm}\label{falconermain}
If $F$ is a $G_\delta$ set in $\R$, then that $F$ is in the class $\mathcal G^s$ is equivalent to that there exists a constant $c>0$ such that 
\begin{align}
M_\infty^s(F\cap I)\geq c |I|^s \label{coveringcondition}
\end{align}
for any $I\subset \R$ of the form $[2^km,2^k(m+1))$ where $k\in \mathbb Z$.
\end{thm}

\noindent It is obvious from the definition that if $F\in \mathcal G^s$, then $\dim_H(F)\geq s$. Falconer also proves that

\begin{thm}\label{intersectionproperty}
The class $\mathcal G^s$ is closed under countable intersections.
\end{thm}

We note that sets like $\limsup_{n \to \infty} G_{\bar p}^{f,m}(n,\epsilon)$ are all subsets of $[0,1)$. There is no way that any of these sets can be in the class $\mathcal G^s$ since such sets must be dense in $\mathbb R$. But we defined the sets $G_{\bar p}^{f,m}(n,\epsilon)$ by expanding ${x\in [0,1)}$ using the function $f$. It is clear that we can do similarly in any interval $[n,n+1)$ where $n\in \mathbb Z$. We can thereby extend our sets $\limsup_{n\to \infty}G_{\bar p}^{f,m}(n,\epsilon)$ into $\mathbb R$. Let $\tilde F \subset [0,1)$ be some set of the type $\limsup_{n\to \infty}G_{\bar p}^{f,m}(n,\epsilon)$ and let $F$ be its extension to $\mathbb R$. To make sure that $F_k$ satisfies condition (\ref{coveringcondition}) of Theorem \ref{falconermain}, it is clearly enough to prove that for some constant $c>0$ it holds that
\[
M^s(\tilde F \cap I)\geq c |I|^s  
\]
for all dyadic intervals $I \subset [0,1)$. This would imply that $F \in \mathcal G^s$ and that we can control the dimension of its intersections with other sets in $\mathcal G^s$. Now, the intersections of sets of the type $\tilde F$ are just restrictions to $[0,1)$ of intersections of sets of the type $F_k$. This means that in $[0,1)$, the set $\tilde F$ behaves just like the set $F\in \mathcal G^s$ does in $\mathbb R$. Thus, in the remainder of this text we will say that a set $\tilde F \subset [0,1)$ is in $\mathcal G^s$ if we get a set in $\mathcal G^s$ by extending $\tilde F$ to a set $F$ in the way described above.

\section{Proofs}

\subsection{Functions generating full shifts}\label{fullshiftproof}
In this section we prove Theorems  \ref{nonnormalthm}, \ref{extremelynonormal} and \ref{intersectiontheorem}. When working with sets like $G_{\bar p}^{f,m}(n,\epsilon)$ it is much easier to consider covers consisting only of cylinders from the expansion by $f$ rather than using the dyadic intervals of the outer measure $M^s_\infty$. Let $N^s$ be the outer measure defined as
\[
N_\infty^s(F)=\inf \Big\{\, \sum_{i=1}^\infty |C_i|^s:F\subset \cup_{i=1}^\infty C_i \, \Big\}
\]
where each $C_i$ is a cylinder with respect to the expansion by $f$.

\begin{lem}\label{factorbetweenmeasures}
For each $f$ as described in section \ref{fullshift},  satisfying condition $(i)$, and any set $A\subset [0,1)$ we have
\[
M_\infty^s(A)\geq \frac{1}{2 g_f K_f} N_\infty^s(A).
\]
\end{lem}

\begin{proof}
Given a set $A$, let $(U_i)$ be a cover of $A$ by dyadic cylinders. Let $k$ be the smallest generation for which there is a generation $k$ cylinder from the expansion by $f$ contained in $U_i$. Let $C_{x_1\dots x_k}$ be the largest of these cylinders. It is clear that $C_{x_1\dots x_{k-1}}$ covers at least one endpoint of $U_i$. 

If it does not cover the entire $U_i$, let $C_{y_1\dots y_k}$ be the largest generation $k$ cylinder contained in $U_i \setminus C_{x_1\dots x_{k-1}}$. By the minimality of $k$ we know that $C_{y_1\dots y_{k-1}}$ covers the other endpoint of $U_i$. 

Now together, $C_{x_1\dots x_{k-1}}$ and $C_{y_1\dots y_{j-1}}$ cover $U_i$. Indeed, by the minimality of $k$, any cylinder between $C_{x_1\dots x_{k}}$ and $C_{y_1\dots y_{k}}$ must have generation at least $k$. But all such cylinders must belong to some generation $k-1$ cylinder, and since there are none between $C_{x_1\dots x_{k-1}}$ and $C_{y_1\dots y_{k-1}}$, there can be no gap between these two sets.

By condition $(i)$ we have 
\[
 |C_{x_1\dots x_{k-1}}|+|C_{y_1\dots y_{j-1}}| \leq g_f K_f |C_{x_1\dots x_{k}}|+ g_f K_f|C_{y_1\dots y_{j}}| \leq 2 g_f K_f |U_i|.
\]
We can do this for each $i$ so it implies $ \sum_i |U_i|^s \geq \frac{1}{2g_f K_f}N_\infty^s(A)$. Since this holds for all covers we get
\[
M_\infty^s(A)\geq \frac{1}{2 g_f K_f}N_\infty^s(A).
\]
\end{proof}

The following lemma is a version of less Lemma 1 from \cite{falconer} with $N^s$ instead of $M^s$. The proof is almost identical.

\begin{lem}\label{cylindersandintervals}
Let $F \subset [0,1)$ and $0<c\leq 1$. If $I=[a,b)\subset [0,1)$ is such that 
\[
N_\infty^s(F\cap C)\geq c |C|^s
\]
for all cylinders $C$ with respect to the expansion by $f$ contained in $I$, then
\[
N_\infty^s(F\cap I)\geq c |I|^s.
\]
\end{lem}

The following lemma is a modified version of Lemma 7 in \cite{falconer}. 

\begin{lem}\label{oldfalconerlemma}
Let $\{F_k\}_{k=1}^\infty$ be a sequence of open subsets of $\mathbb R$ such that for some $0<s\leq 1$ and $c>0$ we have that
\[
\lim_{k\to \infty} N_\infty^s(F_k\cap C)\geq c |C|^s 
\]
for every cylinder $C$ with respect to the expansion by $f$. Then 
\[
\bigcap_{m=1}^\infty \bigcup_{k=m}^\infty   F_k \in \mathcal G^s.
\]
\end{lem}

\begin{proof}
For each $m\in \mathbb N$ and each cylinder $C$ we have
\[
N_\infty^s(\cup_{k=m}^\infty F_k\cap C)\geq \lim_{k\to \infty} N^s_\infty (F_k \cap C)\geq c |C|^s. 
\]
By Lemma \ref{cylindersandintervals} we have 
\[
N_\infty^s(\cup_{k=m}^\infty F_k\cap I)\geq \lim_{k\to \infty} N^s_\infty (F_k \cap I)\geq c |I|^s. 
\]
for all dyadic intervals $I$. By Lemma \ref{factorbetweenmeasures} we get 
\[
M_\infty^s(\cup_{k=m}^\infty F_k\cap I)\geq \frac{1}{2g_f K_f}N_\infty^s(\cup_{k=m}^\infty F_k\cap I)\geq \frac{c}{2g_f K_f} |I|^s
\]
for all dyadic intervals $I$. Then by Theorem \ref{falconermain} we have $\cup_{k=m}^\infty F_k \in \mathcal G^s$ and by Theorem \ref{intersectionproperty} we get
\[
\bigcap_{m=1}^\infty \bigcup_{k=m}^\infty  F_k \in \mathcal G^s.
\]
\end{proof}

We will use Lemma \ref{oldfalconerlemma} to prove the theorems. But to be able to apply Lemma \ref{oldfalconerlemma} we need to prove a couple of propositions. 

\begin{prop}\label{unitintervaltocylinders}
Let $0<s\leq 1$ and $f$ be a function satisfying condition $(i)$. If there is a subsequence $\{M_k\}_{k=1}^\infty$ of the natural numbers and $0\leq c<1$ such that 
\[
N_\infty^s([0,1]\cap G_{\bar p}^{f,m}(M_k,\epsilon))<  \frac{c}{K_f^s} 
\]
for all $k$, then for each cylinder $C\subset [0,1)$, with respect to the expansion by $f$, there exists a number $K_{C}$ such that 
\[
N_\infty^s(C\cap G_{\bar p}^{f,m}(M_k, \epsilon/2))< c|C|^s 
\]
for all $k>K_{C}$.
\end{prop}

\begin{proof}
Let $C$ be a cylinder of generation $n$ and consider the set 
\[
C\cap G_{\bar p}^{f,m}\Big(M_k+n,\epsilon-\frac{n}{M_k}\Big).
\]
To cover this set we need to cover a family of generation $M_k+n$ cylinders. For each of these cylinders, we get a generation $M_k$ cylinder in $[0,1] \cap G_{\bar p}^{f,m}(M_k,\epsilon)$ if we remove the first $n$ symbols in the coding. Indeed, these $n$ symbols cannot affect the frequency more than by a term $\frac{n}{M_k}$. By assumption, there is a cover $(U_i)_{i=1}^\infty$ of $[0,1] \cap G_{\bar p}^{f,m}(M_k,\epsilon)$ with value less than $\frac{c}{K_f}$. For each $U_i$, there is a corresponding interval $\tilde U_i$ in $C$ such that $\tilde U_i$ covers the generation $M_k+n$ cylinders in $C$ that corresponds to the generation $M_k$ cylinders that $U_i$ covers. Since we have bounded distortion, we know that $|\tilde U_i|\leq K_f |C||U_i|$. But $(\tilde U_i)_{i=1}^\infty$ is a cover of $C\cap G_{\bar p}^{f,m}(M_k+n,\epsilon-\frac{n}{M_k})$ so
\[
N_\infty^s(C\cap G_{\bar p}^{f,m}(M_k, \epsilon/2))< K_f^s|C|^s N_\infty^s([0,1] \cap G_{\bar p}^{f,m}(M_k, \epsilon)) \leq c|C|^s
\] 
as long as $k$ is so large that $\frac{n}{M_k}<\frac{\epsilon}{2}$. Thus, for any cylinder $C$, there exists a $K_C$ such that $N_\infty^s\big(C\cap G_{\bar p}^{f,m}(M_k,\epsilon/2)\big)<c|C|^s$ for all $k>K_C$.
\end{proof}

\begin{prop}\label{rightvalueofunitinterval}
Let $f$ be a function satisfying condition $(i)$. Then, for any $m$, $\bar p$ and $\epsilon>0$ it holds that
\[
\lim_{n\to \infty} N_\infty^s([0,1] \cap G_{\bar p}^{f,m}(n,\epsilon))\geq \frac{1}{K_f^s} 
\]
for all $s$ such that $0\leq s<\dim_H(G_{\bar p}^{f,m})$.
\end{prop}

\begin{proof}
Assume on the contrary that there exists a $c<1$ and a subsequence $\{M_k\}_{k=1}^\infty$ of the natural numbers such that $N_\infty^s([0,1) \cap G_{\bar p}^{f,m}(M_k,\epsilon))< \frac{c}{K_f^s}$ for all $k$. Then by Proposition \ref{unitintervaltocylinders} we have that for any cylinder $C$ it holds that $N_\infty^s(C \cap G_{\bar p}^{f,m}(M_k,\epsilon/2))< c|C|^s$ for all $k>K_{C}$.

In that case, there is a finite cover by cylinders $\{C_i^1\}_i$ of 
\[
[0,1)\cap G_{\bar p}^{f,m}(M_{k_1},\epsilon/2)
\]
such that $\sum_i |C_i|^s<\frac{c}{K_f}$. Indeed, to attain the value 
\[
N_\infty^s([0,1)\cap G_{\bar p}^{f,m}(M_{k_1},\epsilon/2))<\frac{c}{K_f}
\]
one only has to look among a finite number of covers, all consisting of cylinders of generation at most $M_{k_1}$. Using higher generation cylinders does not give a lower value since $s<1$.

Now, choose $k_2>\max_i\{K_{C_i^1}\}$. Then there are finite covers $\{C_{i,j}\}_j$ of $C_i^1\cap G_{\bar p}^{f,m}(M_{k_2},\epsilon/2)$ such that $\sum_j |C_{i,j}|^s<c|C_i^1|^s$ for all $i$. Let $\{C_i^2\}$ be all the covers of the $C_i^1$ together. This is a cover of 
\[[0,1)\cap G_{\bar p}^{f,m}(M_{k_1},\epsilon/2)\cap G_{\bar p}^{g,m}(M_{k_2},\epsilon/2)
\]
and its value is at most
\[
\sum_i |C_i^2|^s<\sum_i c|C_i^1|^s <\frac{c^2}{K_f^s}.
\]
Continuing in this way we get that 
\[
[0,1) \cap \big(\cap_{i=1}^{\infty}G_{\bar p}^{f,m}(M_{k_i},\epsilon/2)\big)\supset \cap_{n=M_{k_1}}^{\infty}G_{\bar p}^{f,m}(n,\epsilon/2)
\]
can be covered by covers consisting of arbitrarily small cylinders and with arbitrarily small value. It now follows that
\[
\dim_H(\cap_{n=M_{k_1}}^{\infty}G_{\bar p}^{f,m}(n,\epsilon/2))\leq s
\]
for any $k_1>K_{[0,1]}$. But 
\[
\bigcup_{M=1}^{M_{k_1}}\bigcap_{n=M}^\infty G_{\bar p}^{f,m}(n,\epsilon/2))=\bigcap_{n=M_{k_1}}^\infty G_{\bar p}^{f,m}(n,\epsilon/2))
\]
so 
\[
\dim_H(\cup_{M=1}^\infty \cap_{n=M}^{\infty}G_{\bar p}^{f,m}(n,\epsilon/2))\leq s.
\]
Since $G_{\bar p}^{f,m}\subset \bigcup_{M=1}^{\infty} \bigcap_{n=M}^\infty G_{\bar p}^{f,m}(n,\epsilon/2)$, this implies that 
\[
\dim_H(G_{\bar p}^{f,m})\leq \dim_H \Big(\bigcup_{M=1}^{\infty} \bigcap_{n=M}^\infty G_{\bar p}^{f,m}(n,\epsilon/2))\Big)\leq s,
\]
which is a contradiction.
\end{proof}

\begin{prop}\label{rightvalueforcylinders}
Let $f$ be a function satisfying condition $(i)$. Then for each $m$, $\epsilon>0$, $\bar p$ and each cylinder $C\subset [0,1)$,  it holds that
\[
\lim_{n\to \infty} N_\infty^s(C \cap G_{\bar p}^{f,m}(n,\epsilon))\geq \frac{1}{K_f^{2s}}|C|^s
\]
for all $s$ such that $0\leq s<\dim_H(G_{\bar p}^{g,m})$.
\end{prop}

\begin{proof}
Let $C$ be a generation $n$ cylinder and let $(U_i)_{i=1}^\infty$ be a cover of the set $C\cap G_{\bar p}^{f,m}(m+n,\epsilon)$. Reasoning as in the proof of Proposition \ref{unitintervaltocylinders} we conclude that there is a corresponding cover $(\tilde U_i)_{i=1}^\infty$ of $[0,1) \cap G_{\bar p}^{f,m}(m,\epsilon-\frac{n}{m})$ such that $|\tilde U_i|\leq \frac{K_f}{|C|}|U_i|$ for all $i$. Using Proposition \ref{rightvalueofunitinterval} we get
\[
\lim_{n\to \infty} N_\infty^s(C \cap G_{\bar p}^{f,m}(n,\epsilon))\geq \frac{|C|^s}{K_f^s}\lim_{n\to \infty} N_\infty^s([0,1] \cap G_{\bar p}^{f,m}(n,\frac{\epsilon}{2})) \geq \frac{1}{K_f^{2s}}|C|^s.
\]
\end{proof}

\begin{prop}\label{maintheorem}
Let $f$ be a function satisfying condition $(i)$ and $m\in \mathbb N$. If $\bar p$ is such that there is a word $v\neq 0^m$ such that $p_v>0$, then if $0<\epsilon<p_v$, then
\[
\bigcap_{N=1}^\infty \bigcup_{n=N}^\infty G_{\bar p}^{f,m}(n,\epsilon)
\]
is in the class $\mathcal G^s$ for each $s\leq \dim_H(G_{\bar p}^{f,m})$.
\end{prop}

\begin{proof}
By Remark \ref{gopen}
\[
\bigcap_{N=1}^\infty \bigcup_{n=N}^\infty G_{\bar p}^{f,m}(n,\epsilon)
\]
is a $G_\delta$ set. Using Proposition \ref{rightvalueforcylinders}, we can now apply Lemma \ref{oldfalconerlemma}. By Remark \ref{dimensionequal} we get the proposition.
\end{proof}

But by Theorem \ref{intersectionproperty} the class $\mathcal G^s$ is closed under intersections so we get

\begin{prop}\label{accumulationthm}
Let $f$ be a function satisfying condition $(i)$ and $m\in \mathbb N$. If $\bar p$ is such that there is a word $v\neq 0^m$ such that $p_v>0$, then for some $K\in \mathbb N$ the set
\[
\bigcap_{k=K}^\infty \limsup_{n\to \infty} G_{\bar p}^{f,m}(n,\frac{1}{k})
\]
is in the class $\mathcal G^s$ for all $s\leq \dim_H(G_{\bar p}^{f,m})$. This means that the set of points $x$ for which $\bar p$ is an accumulation point in $R^{g_f^m}$ of $\Big(\frac{\tau_{w_i}^f(x,n)}{n}\Big)_{i=1}^{g_f^m}$ as $n\to \infty$ is in the  class $\mathcal G^s$ for all $s\leq \dim_H(G_{\bar p}^{f,m})$.
\end{prop}

We are now ready to prove the theorems stated in the introduction.

\begin{proof}[Proof of Theorem \ref{nonnormalthm}]
The set $\bigcap_{k=K}^\infty \limsup_{n\to \infty} G_{\bar p}^{g,m}(n,1/k)$ contains no $m$-normal numbers as long as $p_w \neq C_w$ for all words $w$ of length $m$. So by condition $(ii)$ and Proposition \ref{accumulationthm} the set of numbers that are not $m$-normal to $f$ contains a set from the class $\mathcal G^s$ for every $s<1$. By Theorem \ref{intersectionproperty} the class $\mathcal G^s$ is closed under countable intersections, so Theorem \ref{nonnormalthm} follows. 
\end{proof}

\begin{proof}[Proof of Theorem \ref{extremelynonormal}]
Given a word $w$, condition (ii) assures that there are $\bar p$ and $\bar q$ that differ at the position corresponding to the word $w$ while both $\bigcap_{k=K}^\infty \limsup_{n\to \infty} G_{\bar p}^{g,m}(n,1/k)$ and $\bigcap_{k=K}^\infty \limsup_{n\to \infty} G_{\bar q}^{g,m}(n,1/k)$ are in the class $\mathcal G^s$ for $s$ arbitrarily close to $1$. Intersecting these sets we get points for which the frequency of $w$ does not exist. By Theorem \ref{intersectionproperty} we can intersect between such sets corresponding to different words and still be in the class $\mathcal G^s$. Full dimension now follows immediately.
\end{proof}

\begin{proof}[Proof of Theorem \ref{intersectiontheorem}]
We first note that if for some $\bar p_i$ we have $p_{i,w}=0$ for all words $w$ except $0^{m_i}$, then 
\[
\dim_H \Big(\Big\{\, x: \bar p_i \in A^{f_i, m_i}(x) \, \Big\}\Big)=0,
\]
so the statement of the theorem is trivially satisfied. After noting this, the theorem follows from Proposition \ref{accumulationthm} and Theorem \ref{intersectionproperty}.
\end{proof}

\subsection{$\boldsymbol \beta$-shifts where the expansion of 1 terminates}

In this section we prove Theorems \ref{betaextremelynonormal} and \ref{betaintersectiontheorem}. The methods and most of the proofs in this section are almost identical to those of Section \ref{fullshiftproof}. Let $N^s$ be the outer measure defined as
\[
N_\infty^s(F)=\inf \Big\{\, \sum_{i=1}^\infty |C_i|^s:F\subset \cup_{i=1}^\infty C_i \, \Big\}
\]
where each $C_i$ is a cylinder with respect to the expansion by $f_\beta$. 

\begin{lem}\label{betafactorbetweenmeasures}
For each $\beta$, for which the expansion of $1$ terminates, and any set $A\subset [0,1)$ we have
\[
M_\infty^s(A)\geq \frac{1}{2\beta C_\beta^2} N_\infty^s(A).
\]
\end{lem}

\begin{proof}
Take the proof of Lemma \ref{factorbetweenmeasures} and replace $2g_fK_f$ by $2 \beta C_\beta$.
\end{proof}

The following lemma is a version of Lemma 1 from \cite{falconer} with $N^s$ instead of $M^s$. The proof is almost identical.

\begin{lem}\label{betacylindersandintervals}
Let $F \subset [0,1)$ and $0<c\leq 1$. If $I=[a,b)\subset [0,1)$ is such that 
\[
N_\infty^s(F\cap C)\geq c |C|^s
\]
for all cylinders $C$ with respect to the expansion by $f_\beta$ contained in $I$, then
\[
N_\infty^s(F\cap I)\geq c |I|^s.
\]
\end{lem}

\begin{lem}\label{betaoldfalconerlemma}
Let $\{F_k\}_{k=1}^\infty$ be a sequence of open subsets of $\mathbb R$ such that for some $0<s\leq 1$ and $c>0$ we have that
\[
\lim_{k\to \infty} N_\infty^s(F_k\cap C)\geq c |C|^s 
\]
for every cylinder $C$ with respect to the expansion by $f_\beta$. Then 
\[
\bigcap_{m=1}^\infty \bigcup_{k=m}^\infty   F_k \in \mathcal G^s.
\]
\end{lem}

\begin{proof}
Take the proof of Lemma \ref{oldfalconerlemma}, replace $2g_f K_f$ by $2 \beta C_\beta$ and replace the reference to \ref{cylindersandintervals} by \ref{betacylindersandintervals}.
\end{proof}

We will use Lemma \ref{betaoldfalconerlemma} to prove the theorems. But to be able to apply Lemma \ref{betaoldfalconerlemma} we have to prove some propositions.

First we note that among the cylinders with respect to $f_\beta$, there are many that are scalings of $[0,1)$. Indeed, for any cylinder $C_{x_0 \dots x_{n-1}}$ there is a maximal number $l$ such that $C_{x_0 \dots x_{n-1}}=C_{x_0 \dots x_{n-1}0^l}$. By the maximality of $l$, we know that the cylinder $C_{x_0 \dots x_{n-1}0^l 1}$ is nonempty. This implies that any word in $S_\beta$ can follow after $x_0 \dots x_{n-1} 0^{l+1}$ and thereby $f^{n+l+1}(C_{x_0 \dots x_{n-1}0^{l+1}})=[0,1)$, \mbox{i.e.} $C_{x_0 \dots x_{n-1}0^{l+1}}$ is a scaling of $[0,1)$ by a factor $\beta^{-(n+l+1)}$. We also note that 
\[
|C_{x_0 \dots x_{n-1} 0^{l+1}}|\geq \frac{1}{\beta}|C_{x_0 \dots x_{n-1}}|.
\]

\begin{prop}\label{betaunitintervaltospecialcylinders}
Let $0<s\leq 1$ and $\beta\in (1,2)$ be such that the expansion of $1$ terminates. If there is a subsequence $\{M_k\}_{k=1}^\infty$ of the natural numbers and $0\leq c<1$ such that 
\[
N_\infty^s([0,1]\cap G_{\bar p}^{\beta,m}(M_k,\epsilon))< c 
\]
for all $k$, then for each cylinder $C$ with respect to $f_\beta$ such that $C$ is a scaling of $[0,1)$, there exists a number $K_{C}$ such that 
\[
N_\infty^s(C\cap G_{\bar p}^{\beta,m}(M_k, \epsilon/2))< c|C|^s 
\]
for all $k>K_{C}$.
\end{prop}

\begin{proof}
Take the proof of Proposition \ref{unitintervaltocylinders} and replace the use of ''any cylinder'' by ''any cylinder that is a scaling of $[0,1)$'', and replace $K_f$ by $1$.
\end{proof}

\begin{prop}\label{betaunitintervaltocylinders}
Let $0<s\leq 1$ and $\beta\in (1,2)$ be such that the expansion of $1$ terminates. Let $Q(s,\beta)=\frac{1}{\beta^s}\sum_{i=0}^{\infty} (1-\frac{1}{\beta})^{is}$. If there is a subsequence $\{M_k\}_{k=1}^\infty$ of the natural numbers and $0\leq c<1$ such that 
\[
N_\infty^s([0,1]\cap G_{\bar p}^{\beta,m}(M_k,\epsilon))< \frac{c}{2Q(s,\beta)} 
\]
for all $k$, then for each cylinder $C$ with respect to $f_\beta$, there exists a number $K_{C}$ such that 
\[
N_\infty^s(C\cap G_{\bar p}^{\beta,m}(M_k, \epsilon/2))< c|C|^s 
\]
for all $k>K_{C}$.
\end{prop}

\begin{proof}
Let $C_{x_0 \dots x_{n-1}}$ be a cylinder with respect to $f_\beta$. Then there is a number $l$ such that $C_{x_0\dots x_{n-1}}=C_{x_0 \dots x_{n-1}0^l}$ and $|C_{x_0 \dots x_{n-1}0^{l+1}}|\geq \frac{1}{\beta}|C_{x_0 \dots x_{n-1}}|$, where $C_{x_0 \dots x_{n-1}0^{l+1}}$ is a scaling of $[0,1)$. Starting with the cylinder $C_{x_0 \dots x_{n-1} 0^{l}1}$ we can repeat this argument. Repeating it $j$ times we can split $C_{x_0 \dots x_{n-1}}$ into $j+1$ cylinders, all but the last being scalings of $[0,1)$. Covering $C_{x_0 \dots x_{n-1}}$ with these we get at most the value
\[
|C_{x_0 \dots x_{n-1}}|^s\frac{1}{\beta^s}\sum_{i=0}^{j} (1-\frac{1}{\beta})^{is}<|C_{x_0 \dots x_{n-1}}|^sQ(s,\beta)<\infty
\]
for all $j$. Given $\delta>0$ we can choose $j$ so large that the cylinder not being a scaling of $[0,1)$ contributes less than $\delta$ to the value of the cover. By the assumptions of this proposition and by Proposition \ref{betaunitintervaltospecialcylinders}, we can find a constant $K_{C_{x_0 \dots x_{n-1}}}$ such that for any $C_i$ of the $j$ cylinders that are scalings of $[0,1)$ we have 
\[
N_\infty^s(C_i\cap G_{\bar p}^{\beta,m}(M_k, \epsilon/2))< \frac{c}{2Q(s,\beta)}|C|^s 
\]
for all $k>K_{C_{x_0 \dots x_{n-1}}}$. Since $\sum_{i=1}^{j} |C_i|^s \leq |C_{x_0 \dots x_{n-1}}|^s Q(s, \beta)$ we get 
\begin{align*}
N_\infty^s(C_{x_0 \dots x_{n-1}}\cap G_{\bar p}^{\beta,m}(M_k, \epsilon/2))<\,\, & |C_{x_0 \dots x_{n-1}}|^s Q(s, \beta) \frac{c}{2Q(s,\beta)}+\delta \\  <\, \, & c|C_{x_0 \dots x_{n-1}}|^s 
\end{align*}
for all $k>K_{C_{x_0 \dots x_{n-1}}}$ if $\delta>0$ was chosen small enough.
\end{proof}

\begin{prop}\label{betarightvalueofunitinterval}
Let $\beta\in (1,2)$ be such that the expansion of $1$ terminates. Then, for any $m$, $\bar p$ and $\epsilon>0$ it holds that
\[
\lim_{n\to \infty} N_\infty^s([0,1) \cap G_{\bar p}^{\beta,m}(n,\epsilon))\geq \frac{1}{2Q(s,\beta)} 
\]
for all $s$ such that $0\leq s<\dim_H(G_{\bar p}^{\beta,m})$.
\end{prop}

\begin{proof}
Take the proof of Proposition \ref{rightvalueofunitinterval}. Replace the set $G_{\bar p}^{f,m}(M_k,\epsilon))$ by $G_{\bar p}^{\beta,m}(M_k,\epsilon)$, replace $G_{\bar p}^{f,m}$ by $G_{\bar p}^{\beta,m}$, replace $K_f$ by $2Q(s,\beta)$ and replace the reference to Proposition \ref{unitintervaltocylinders} by a reference to Proposition \ref{betaunitintervaltocylinders}.
\end{proof}

\begin{prop}\label{betarightvalueforcylinders}
Let $\beta\in (1,2)$ be such that the expansion of $1$ terminates. Then for each $m$, $\epsilon>0$, $\bar p$ and each cylinder $C\subset [0,1)$,  it holds that
\[
\lim_{n\to \infty} N_\infty^s(C \cap G_{\bar p}^{\beta,m}(n,\epsilon))\geq \frac{1}{2\beta^s Q(s,\beta)}|C|^s
\]
for all $s$ such that $0\leq s<\dim_H(G_{\bar p}^{\beta,m})$.
\end{prop}

\begin{proof}
Let $C$ be a cylinder with respect to $f_\beta$. Then it contains a cylinder $C^*$ such that $|C^*|\geq \frac{|C|}{\beta}$ and such that $C^*$ is a scaling of $[0,1)$. Let the generation of $C^*$ be $n$ and let $(U_i)_{i=1}^\infty$ be a cover of $C^*\cap G_{\bar p}^{\beta,m}(m+n,\epsilon)$. Reasoning as in the proof of Proposition \ref{unitintervaltocylinders}, replacing $K_f$ by 1, we conclude that there is a corresponding cover $(\tilde U_i)_{i=1}^\infty$ of $[0,1) \cap G_{\bar p}^{\beta,m}(m,\epsilon-\frac{n}{m})$ such that $|\tilde U_i|\leq \frac{1}{|C^*|}|U_i|$ for all $i$. Using Proposition \ref{betarightvalueofunitinterval} we get
\begin{align*}
& \lim_{n\to \infty} N_\infty^s(C \cap G_{\bar p}^{\beta,m}(n,\epsilon))\geq \lim_{n\to \infty} N_\infty^s(C^* \cap G_{\bar p}^{\beta,m}(n,\epsilon))\geq \\
&|C^*|^s\lim_{n\to \infty} N_\infty^s([0,1) \cap G_{\bar p}^{\beta,m}(n,\frac{\epsilon}{2})) \geq \frac{1}{2Q(s,\beta)}|C^*|^s \geq \frac{1}{2\beta^s Q(s,\beta)}|C|^s.
\end{align*}
\end{proof}

\begin{prop}\label{betamaintheorem}
Let $\beta\in (1,2)$ be such that the expansion of $1$ terminates and $m\in \mathbb N$. If $\bar p$ is such that there is a word $v\neq 0^m$ such that $p_v>0$, then if $0<\epsilon<p_v$, then
\[
\bigcap_{N=1}^\infty \bigcup_{n=N}^\infty G_{\bar p}^{\beta,m}(n,\epsilon)
\]
is in the class $\mathcal G^s$ for each $s\leq \dim_H(G_{\bar p}^{\beta,m})$.
\end{prop}

\begin{proof}
By Remark \ref{betagopen}
\[
\bigcap_{N=1}^\infty \bigcup_{n=N}^\infty G_{\bar p}^{\beta,m}(n,\epsilon)
\]
is a $G_\delta$ set. Using Proposition \ref{betarightvalueforcylinders}, we can now apply Lemma \ref{betaoldfalconerlemma}. By Remark \ref{dimensionequal} we get the proposition.
\end{proof}

But by Theorem \ref{intersectionproperty} the class $\mathcal G^s$ is closed under intersections so we get

\begin{prop}\label{betaaccumulationthm}
Let $\beta\in (1,2)$ be such that the expansion of $1$ terminates and $m\in \mathbb N$. If $\bar p$ is such that there is a word $v\neq 0^m$ such that $p_v>0$, then for some $K\in \mathbb N$ the set
\[
\bigcap_{k=K}^\infty \limsup_{n\to \infty} G_{\bar p}^{\beta ,m}(n,\frac{1}{k})
\]
is in the class $\mathcal G^s$ for all $s\leq \dim_H(G_{\bar p}^{\beta ,m})$. This means that the set of points $x$ for which $\bar p$ is an accumulation point in $R^{2^m}$ for $\Big(\frac{\tau_{w_i}^\beta(x,n)}{n}\Big)_{i=1}^{2^m}$ as $n\to \infty$ is in the  class $\mathcal G^s$ for all $s\leq \dim_H(G_{\bar p}^{\beta,m})$.
\end{prop}

We are now ready to prove the theorems stated in the introduction.

\begin{proof}[Proof of Theorem \ref{betaextremelynonormal}]
Let $w$ be a word of length $m$. It follows from higher-dimensional multifractal analysis on subshifts of finite type (see \cite{barreira}) that there is  vector $\bar p$ such that $\dim_H(G_{\bar p}^{\beta ,m})=1$ and a vector $\bar q$ with $p_w\neq q_w$ such that $\dim_H(G_{\bar p}^{\beta ,m})$ is arbitrarily close to 1. Thus, by Proposition \ref{betaaccumulationthm} both the set $\bigcap_{k=K}^\infty \limsup_{n\to \infty} G_{\bar p}^{g,m}(n,1/k)$ and $\bigcap_{k=K}^\infty \limsup_{n\to \infty} G_{\bar q}^{g,m}(n,1/k)$ are in the class $\mathcal G^s$ for $s$ arbitrarily close to $1$. Intersecting these sets we get points for which the frequency of $w$ does not exist. By Theorem \ref{intersectionproperty} we can intersect between such sets corresponding to different words and still be in the class $\mathcal G^s$. Full dimension now follows immediately.
\end{proof}

\begin{proof}[Proof of Theorem \ref{betaintersectiontheorem}]
We first note that if for some $\bar p_i$ we have $p_{i,w}=0$ for all words $w$ except $0^{m_i}$, then 
\[
\dim_H \Big( \Big\{\, x :\bar p_i \in A^{\beta_i, m_i}(x) \, \Big\}\Big)=0,
\]
so the statement of the theorem is trivially satisfied. After noting this, the theorem follows from Proposition \ref{betaaccumulationthm} and Theorem \ref{intersectionproperty}.
\end{proof}


\begin{thebibliography}{77}

\bibitem{barreira}
L. Barreira, B. Saussol and J. Schmeling, {\em Distribution of frequencies of digits via multifractal analysis}, J. Number Theory 97 (2002), 410--438.

\bibitem{falconer}
K. Falconer, {\em Sets with large intersection properties}, J. London Math. Soc. (2) 49 (1994), no. 2, 267--280.


\bibitem{Parry} W. Parry, {\em On the $\beta$-expansion of real
    numbers}, Acta Mathematica Academiae Scientiarum Hungaricae, 11 (1960), 401--416.

\bibitem{Renyi} A. R\'enyi, {\em Representations for real numbers and their ergodic properties}, Acta Mathematica Academiae Scientiarum Hungaricae, 8 (1957), 477--493.

\end{thebibliography}
\end{document}